\documentclass[12pt,a4paper,article]{amsart}
\usepackage{amssymb,amsthm,amsmath} 
\usepackage[T1]{fontenc}
\usepackage[utf8]{inputenc}  
\usepackage{enumerate}     	
\usepackage{graphicx}
\usepackage{bbm}
\usepackage{amscd}
\usepackage{setspace}
\usepackage{mathabx}
\usepackage{amsfonts}
\usepackage[margin=.45in, font=footnotesize,labelfont=bf,labelsep=period]{caption}
\usepackage{datetime}
\usepackage{a4wide}
\usepackage{float}

\theoremstyle{definition}
 
\newtheorem{proposition}{Proposition}[]
\newtheorem{lemma}{Lemma}[]

\newtheorem{remark}{Remark}[] 
\newtheorem{assumption}{Assumption}[]

\title[Two-sided Poisson control of diffusions]{Two-sided Poisson control of linear diffusions}
\author{Harto Saarinen}
\address{Harto Saarinen, Department of Mathematics and Statistics, University of Turku, FI - 20014 Turun Yliopisto, Finland, \texttt{hoasaa@utu.fi}}
\date{\today, \currenttime}

\begin{document}

\begin{abstract}
We study a class of two-sided optimal control problems of general linear diffusions under a so-called Poisson constraint: the controlling is only allowed at the arrival times of an independent Poisson signal processes. We give a weak and easily verifiable set of sufficient conditions under which we derive a quasi-explicit unique solution to the problem in terms of the minimal $r$-excessive mappings of the diffusion. We also investigate limiting properties of the solutions with respect to the signal intensity of the Poisson process. Lastly, we illustrate our results with an explicit example.


\smallskip
\noindent \textup{2010} \textit{Mathematics Subject Classification}: \textup{93E20, 60J60}

\end{abstract}

\maketitle

\section{Introduction}

In a classical stochastic singular control problems the objective is to maximise the expected discounted cumulative yield, given by a function of a stochastic process, and the maximisation takes place over controls that can be continuously applied. In these problems, the controlling can be allowed only downwards or upwards (one-sided problems) or both (two-sided problems) depending on the application. Both one-sided and two-sided singular problems have been studied extensively due to their mathematical attractiveness and applicability in various fields. These include, for example, a reversible or irreversible investment problems (cf. \cite{Oksendal2000, ChiarollaHaussman2005, GuoPham2005}), where the controls can be interpreted as investor purchasing capital and possibly selling it, and rational harvesting (cf. \cite{LunguOksendal1997, AlvarezHening2019}), where the controls can be seen as harvesting and replanting.

In Poisson optimal control problems, also called constrained control problems (these terminologies are from \cite{Lange2020} and \cite{MenaldiRobin2017}), the potential control opportunities are restricted to jump times of an independent Poisson process. The main motivation for introducing Poisson control problem is that in applications it is not often possible to apply the control continuously. For example, in \cite{RogersZane2002} and \cite{Matsumoto2006} due to liquidity effects when trading financial assets and in \cite{Lempa2014} the possibility of modelling imperfect flow of information is considered. Nowadays the literature on Poisson type control problems is quite extensive. Some examples include optimal stopping \cite{DupuisWang2002, Lempa2012}, stopping games \cite{LiangSun2019, LiangSun2020}, ergodic control \cite{Wang2001, LempaSaarinen2021}, optimal switching \cite{LiangWei2016}, extensions to inhomogeneous Poisson processes \cite{HobsonZeng2019, Hobson2021} and more general signal processes \cite{MenaldiRobin2017, MenaldiRobin2016, MenaldiRobin2018}. 

In the classical singular control problem when the controlling is costless the optimal policy is often a local time type reflecting barrier policy. The effect of introducing the Poisson restriction to control opportunities has usually similar effect on the optimal strategy, although for different reasons as introducing a constant transaction cost to the model. Indeed, when the controlling is costly the optimal strategy is a sequential impulse control, where the decision maker chooses a sequence of stopping times $\{ \tau_1, \tau_2, \ldots \}$ to exert the control and corresponding impulse sizes $\{ \zeta_1, \zeta_2, \ldots \}$, see \cite{AlvarezLempa2008, Alvarez2004, HarrisonSellkeTaylor1983}. A similar strategy is optimal in the Poisson control problems, but the possible intervention times are restricted exogenously by the signal process \cite{Wang2001, Lempa2014, LempaSaarinen2021}.

We extend the Poisson control framework in the following way. Our problem setting is closely related to \cite{Wang2001}, \cite{Matomaki2012} and \cite{Lempa2014}. Similarly to \cite{Matomaki2012} and \cite{Lempa2014}, we assume that the underlying dynamics follow a general one-dimensional diffusion. In \cite{Wang2001}, the Poisson control problem was first introduced in the literature by considering controlling of standard Brownian motion with quadratic payoff at the jump times of independent Poisson process. These results were then extended to a more general underlying and payoff structure by \cite{Lempa2014}, where one-sided problem is considered. In \cite{Matomaki2012}, a two-sided singular control problem is solved using the techniques from classical theory of linear diffusions and $r$-excessive mappings, which lead to quasi-explicit solution. Our study extends the findings of \cite{Lempa2014} by considering two-sided Poisson control policies using similar, rather easily verifiable set of assumptions as in \cite{Matomaki2012}. Further, our results can also be seen as a generalization to those of \cite{Matomaki2012}, since the results coincide when the Poisson arrival rate approaches infinity. To the best of our knowledge these results are new.

The remainder of the study is organized as follows. In Section 2 we set up the underlying dynamics and formulate the two-sided Poisson control problem. Main assumptions and auxiliary calculations are done in Section 3, whereas in Section 4 we derive a candidate solution for the problem and verify its optimality. Asymptotic results connecting the problem to the singular problem are proved in Section 5. The study is concluded by considering explicit example in Section 6.

\section{Underlying dynamics and problem setting}

Let $(\Omega, \mathcal{F}, \{\mathcal{F}_t\}_{t \geq 0}, \mathbb{P})$ be a filtered probability space that satisfies the usual conditions. We consider an uncontrolled process $X_t$ defined on $(\Omega, \mathcal{F}, \{\mathcal{F}_t\}_{t \geq 0}, \mathbb{P})$, which lives in $\mathbb{R}_+$, and is given as a strong solution to a regular Itô diffusion
\begin{equation*}
dX_t=\mu(X_t)dt+ \sigma(X_t)dW_t, \quad \quad X_0 = x,
\end{equation*}
where $W_t$ is the Wiener process and the functions $\mu: \mathbb{R}_+ \to \mathbb{R}$ and $\sigma:\mathbb{R}_+ \to \mathbb{R}_+$ are sufficiently smooth (see e.g. \cite{KaratzasShreve1991} chapter 5). The boundaries of the state space $\mathbb{R}_+$ are assumed to be natural. Even though we consider the case where the process evolves in $\mathbb{R}_+$, we remark that the results would remain the same with obvious changes even if the state space would be replaced with any interval. 

As usual, we define the second-order linear differential operator $\mathcal{A}$ which represents the infinitesimal generator of the diffusion $X$ as
\begin{equation*}
\mathcal{A} = \mu(x) \frac{d}{dx} + \frac{1}{2} \sigma^2(x)\frac{d^2}{dx^2},
\end{equation*}
and for a given $r > 0$ we respectively denote the increasing and decreasing solutions to the differential equation $(\mathcal{A}-r)f=0$ by $\psi_r > 0$ and $\varphi_r > 0$. These solutions are called \emph{fundamental solutions} or \emph{minimal excessive mappings} (\cite{BorodinSalminen} p.19, p.33).

 For $r>0$, we denote by $\mathcal{L}_1^r$ the set of functions $f$ on $\mathbb{R}_+$, which satisfy the integrability condition 
 \begin{equation*}
     \mathbb{E}_x \bigg[ \int_0^{\infty}e^{-r s} |f(X_s)| ds \bigg] < \infty.
 \end{equation*}
For any $f \in \mathcal{L}_1^r$, we define the functional $(R_r f): \mathbb{R}_+ \to \mathbb{R}$ by
\begin{equation*}
(R_r f)(x)=\mathbb{E}_x \bigg[ \int_0^{\infty}e^{-r s} f(X_s) ds \bigg].
\end{equation*}
This functional, called the \emph{resolvent}, is the inverse of the differential operator $r-\mathcal{A}$.

Also define the scale density of the diffusion by
\begin{equation*}
S'(x) = \exp \bigg( - \int^x \frac{2 \mu(z)}{\sigma^2(z)}dz \bigg),
\end{equation*}
which is the derivative of the monotonic (and non-constant) solution to the differential equation $\mathcal{A} S=0$. Given the scale density and the fundamental solutions, the resolvent $(R_r f)(x)$ can be re-expressed as (\cite{BorodinSalminen} p.19)
\begin{align*}
    (R_r f)(x) =  B_r^{-1} \varphi_r(x)\int_0^{x} \psi_r(y)f(y)m'(y)dy 
      +  B_r^{-1} \psi_r(x)\int_x^{\infty} \varphi_r(y) f(y) m'(y) dy,
\end{align*}
where 
\begin{equation*}
B_r = \frac{\psi_r'(x)}{S'(x)} \varphi_r(x) -\frac{\varphi_r'(x)}{S'(x)} \psi_r(x)
\end{equation*}
is the constant Wronskian (does not depend on $x$) and 
\begin{equation*}
m'(x) = \frac{2}{\sigma^2(x)S'(x)}
\end{equation*}
is the density of the speed measure. We also recall the resolvent equation (\cite{BorodinSalminen} p. 4)
\begin{equation} \label{Eq:ResolventEquation} 
R_q R_r= \frac{R_r-R_q}{q-r},
\end{equation}
where $q>r>0$.

Having setup the underlying dynamics we next describe the control problem. We study a maximization problem of the expected value of the cumulative payoff when the controlling of $X$ is allowed only at the jump times of a signal process $N$.
\begin{assumption}
    The process $N$ is assumed to be Poisson process with parameter $\lambda$, that is independent of $X$. Further, the filtration $\{\mathcal{F}_t\}_{t \geq 0}$ is augmented such that it is rich enough to carry the Poisson process.
\end{assumption}
We call a control policy $(\zeta^d_t, \zeta^u_t)$ admissible if both processes are non-negative, non-decreasing, right-continuous and can be represented as
\begin{equation*}
    \zeta^d_t = \int_0^{t-} \eta^d_s dN_s, \, \, \quad \zeta^u_t= \int_0^{t-} \eta^u_s dN_s,
\end{equation*}
where $\eta^d$ and $\eta^u$ are $\{F_t\}$-predictable processes. Thus, the controlled process $X^{\zeta}_t$ is given by 
\begin{equation*}
    X^{\zeta}_t = x + \int_0^t \mu(X_s)ds + \int_0^t \sigma(X_s)dW_s - \zeta^d_t + \zeta^u_t.
\end{equation*}
Denote
\begin{equation*} 
J(x, \zeta) = \mathbb{E}_x \bigg[ \int_0^{\infty} e^{-rs} \big( \pi(X^{\zeta}_s)ds +\gamma_d d\zeta_d - \gamma_u d\zeta_u \big) \bigg].
\end{equation*}
Then the optimal control problem is to find the optimal value function
\begin{equation} \label{Eq:TheValueFunction}
V(x) = \sup_{(\zeta_d, \zeta_u)} J(x, \zeta),
\end{equation}
where the supremum is taken over all admissible controls and $\gamma_d$ and $-\gamma_u$ are constants, called the unit price and unit cost, respectively. The aim is also to characterize semi-explicitly the optimal control policy $(\zeta_d^*, \zeta_u^*)$ that realizes the supremum in (\ref{Eq:TheValueFunction}).

We end this section by stating useful bounds on the value function that hold in given regions of the state space. The lemma follows directly from corollary 2.4 of \cite{Lempa2014}.
\begin{lemma} \label{Lemma:ValueFunctioBounds}
\hspace{1cm}
\begin{enumerate}[$(i)$]
    \item  If there exists an interval $(b,\infty)$ of the state space $\mathbb{R}_+$ such that it is always sub-optimal to use the control $\zeta^u_t$, the value function satisfies 
    \begin{equation*}
        V(x) \leq (R_{r+\lambda}\pi_{\gamma_d})(x)+ \frac{\lambda}{r} \sup_{x \in \mathbb{R}_+} (R_{r+\lambda}\theta_{d})(x)
    \end{equation*}
    when $x > b$.
    \item If there exists an interval $(0,a)$ of the state space $\mathbb{R}_+$, such that it is always sub-optimal to use the control $\zeta^d_t$, the value satisfies
    \begin{equation*}
        V(x) \geq (R_{r+\lambda}\pi_{\gamma_u})(x)+ \frac{\lambda}{r} \inf_{x \in \mathbb{R}_+} (R_{r+\lambda}\theta_{u})(x)
    \end{equation*}
    when $x < a$.
\end{enumerate}
\end{lemma}

\section{Auxiliary results}

To set up the framework further, we denote $n \in \{d,u\}$, and define the functions $\pi_{\gamma_n}: \mathbb{R}_+ \to \mathbb{R}$ and $g_n: \mathbb{R}_+ \to \mathbb{R}$ as
\begin{align}
&g_n(x)=\gamma_n x-(R_r\pi)(x), \label{Eq:gnDefinition} \\ 
&\pi_{\gamma_n}(x) = \lambda \gamma_n x + \pi(x) \label{Eq:PigammaDefinition}.
\end{align}
Also, define the functions $\theta_n: \mathbb{R}_+ \to \mathbb{R}$ as
\begin{equation*} \label{Eq:ThetaDefinition}
\theta_n(x) = \pi(x) + \gamma_n(\mu(x)-rx).
\end{equation*}
In the literature, the function $\theta_n(x)$ is known as the net convenience yield of holding inventories cf. \cite{Alvarez2001, Dixit1990} and is often in a key role when determining the optimal policies  \cite{Alvarez2001, AlvarezLempa2008, Lempa2014}.

The following lemma gives convenient relationships between the defined auxiliary functions. It can be proved by using the resolvent equation (\ref{Eq:ResolventEquation}) and the harmonicity properties of $(R_r\pi)$.
\begin{lemma} \label{Lemma:LinksBetweenFunctions} Let $r > 0$ and $g_n, \pi_{\gamma_n}, \theta_n \in \mathcal{L}_1^r$. Then
\begin{align*}
& g_n(x) = -(R_r \theta_n)(x) \\
& (R_{r+\lambda} \pi_\gamma)(x) = \lambda (R_{r+\lambda}g)(x) + (R_r\pi)(x) \\
& \lambda (R_{r+\lambda}g)(x) = (R_{r+\lambda}\theta_r)(x)+g(x).
\end{align*}
\end{lemma}
We now collect the main assumptions that are needed to prove that the solution to the control problem is well-defined and unique.
\begin{assumption} \label{Ass:MainAssumptions} Assume that:
\begin{enumerate}[(i)]
\item $\gamma_d < \gamma_u$, 
\item the functions $\theta_n$ and id: $x \mapsto x$ are in $\mathcal{L}_1^r$,
\item the payoff $\pi$ is positive, continuous and non-decreasing,
\item $\mu'(x)<r$,
\item there is a unique state $x^*_n \geq 0$ such that $\theta_n'(x) \, \substack{> \\ = \\ <} \, 0$ when $x  \, \substack{< \\ = \\ >} \, x^*_n$,
\item $\theta_n$ satisfies the limiting conditions $\lim_{x \to 0+} \theta_n(x) > 0$ and $\lim_{x \to \infty} \theta_n(x) = - \infty$.
\end{enumerate}
\end{assumption}

Some remarks on these assumptions are in order. The converse of the assumption (i) would easily lead to infinite value functions. Item (ii) guarantees sufficient integrability so that the resolvents of the defined functions exist. The assumptions (iii), (v) and (vi) are in key role when proving the existence of the solution. These type of assumptions are quite standard in stochastic control problems, where explicit solutions are desired, see e.g. \cite{Alvarez2001, Matomaki2012, Lempa2014}. The assumption (iv) can be seen as an upper bound for the Lipschitz constant for the coefficient $\mu$, but it also guarantees together with assumption (ii) that the fundamental solutions $\psi_r$ and $\varphi_r$ are convex, see Corollary 1 in \cite{Alvarez2003}.

\begin{remark} \label{Remark:xu xd Ordering}
It follows from the assumptions \ref{Ass:MainAssumptions} (i), (iv) and (v) that $x_u^* < x^*_d$. To see this, note that for $x > x^*_d$
$$
\theta_u'(x) = \pi'(x)+\gamma_u(\mu'(x)-r) \leq \pi'(x) + \gamma_d(\mu'(x)-r) = \theta_d'(x) \leq 0
$$
\end{remark}

Having stated the main assumptions, we define the functionals $L_f^{r}: \mathbb{R}_+ \to \mathbb{R}$ and $K_f^{r}: \mathbb{R}_+ \to \mathbb{R}$ for any $f \in \mathcal{L}^1_r$ as
\begin{align*}
& L_f^{r}(x) = r\int_x^{\infty} f(y) \varphi_{r}(y)m'(y)dy + \frac{\varphi'_{r}(x)}{S'(x)}f(x), \\
& K_f^{r}(x) = r\int_0^{x} f(y) \psi_{r}(y)m'(y)dy - \frac{\psi'_{r}(x)}{S'(x)}f(x),
\end{align*}
and prove some auxiliary results.
\begin{lemma} \label{Lemma:ResolventRepresentationsLandK}
The functionals $L_f^{r+\lambda}$ and $K_f^{r+\lambda}$ have alternative representations
\begin{align*}
    L_f^{r+\lambda}(x) = & \frac{\sigma^2(x)}{2 \lambda S'(x)}[\varphi_{r+\lambda}''(x) \lambda (R_{r+\lambda}f)'(x)-\varphi_{r+\lambda}'(x) \lambda (R_{r+\lambda}f)''(x)] \\
    K_f^{r+\lambda}(x) = & \frac{\sigma^2(x)}{2 \lambda S'(x)} [\psi'_{r+\lambda}(x) \lambda (R_{r+\lambda} f)''(x)- \psi_{r+\lambda}''(x) \lambda (R_{r+\lambda} f)'(x)]
\end{align*}
\end{lemma}
\begin{proof}
The representation of $L_f^{r+\lambda}$ is proven in lemma 2 of Lempa 2014 and the representation of $K_f^{r+\lambda}$ is proven analogously.
\end{proof}

The next two lemmas provide us with the monotonicity properties of functionals related to $L_f^{r+\lambda}$ and $K_f^{r+\lambda}$, when $f=\theta_n$. The first part of the first lemma is an analogue of the proof of lemma 3 in \cite{Lempa2014} and the second part of lemma 3.1 in \cite{AlvarezLempa2008}.
\begin{lemma} \label{Lemma:LandKsigns} Let the assumptions 1 and 2 hold. Then there exists a unique $\hat{x} < x^*_d$ such that
\begin{equation*}
L_{\theta_d}^{r+\lambda}(x) \substack{> \\ = \\ <} 0, \text{ when } x \substack{< \\ = \\ >} \hat{x}
\end{equation*}
and
\begin{equation*}
J'(x) \substack{< \\ = \\ >} 0, \text{ when } x \substack{< \\ = \\ >} \hat{x},
\end{equation*}
where
\begin{equation*}
J(x)= \frac{(R_{r+\lambda}\pi_{\gamma_d})'(x)-\gamma_d}{\varphi_{r+\lambda}'(x)}.
\end{equation*}
Also, there exists a unique $\tilde{x} > x^*_u$ such that
\begin{equation*}
K_{\theta_u}^{r+\lambda}(x) \substack{> \\ = \\ <} 0, \text{ when } x \substack{> \\ = \\ <} \tilde{x}
\end{equation*}
and
\begin{equation*}
I'(x) \substack{> \\ = \\ <} 0, \text{ when } x \substack{> \\ = \\ <} \tilde{x},
\end{equation*}
where
\begin{equation*}
I(x)= \frac{(R_{r+\lambda}\pi_{\gamma_u})'(x)-\gamma_u}{\psi_{r+\lambda}'(x)}.
\end{equation*}
\end{lemma}

\begin{lemma} \label{Lemma:MonotonicityOfHandQ} The monotonicity of the functions
$$
Q(g_d, \varphi_r; x)=\frac{g_d'(x) L^{r+\lambda}_{\varphi_r}(x)-\varphi_r'(x)L^{r+\lambda}_{g_d}(x)}{\varphi'_{r+\lambda}(x)}
$$
$$
H(\varphi_r, g_u; x)=\frac{\varphi_r'(x)K^{r+\lambda}_{g_d}(x)-g_u'(x) K^{r+\lambda}_{\varphi_r}(x)}{\psi'_{r+\lambda}(x)}
$$
$$
Q(g_d, \psi_r; x)=\frac{g_d'(x)L^{r+\lambda}_{\psi_r}(x)-\psi_r'(x)L^{r+\lambda}_{g_d}(x)}{\varphi'_{r+\lambda}(x)}
$$
$$
H(\psi_r, g_u; x)=\frac{\psi_r'(x)K^{r+\lambda}_{g_u}(x)-g_u'(x)K^{r+\lambda}_{\psi_r}(x)}{\psi'_{r+\lambda}(x)}
$$
is determined by the signs of $L^{r+\lambda}_{\theta_d}$ and $K^{r+\lambda}_{\theta_u}$.
\end{lemma}
\begin{proof}
A straight differentiation and the usage of harmonicity properties of $(R_r\pi)$ and $\varphi_r$ give
$g_d'(x)L^{r+\lambda}_{\varphi}(x)'-\varphi_r'(x) L^{r+\lambda}_{g_d}(x)'=0$. Thus, 
\begin{align*}
& \varphi_{r+\lambda}'(x)^2 Q'(g_d, \varphi_r; x) \\
= & (g_d''(x)L^{r+\lambda}_{\varphi}(x)-\varphi_r''(x) L^{r+\lambda}_{g_d}(x))\varphi_{r+\lambda}'(x) \\
+ & (\varphi_r'(x) L^{r+\lambda}_{g_d}(x)-g_d'(x)L^{r+\lambda}_{\varphi}(x))\varphi_{r+\lambda}''(x) \\
= & L^{r+\lambda}_{g_d}(x)(\varphi_r'(x)\varphi_{r+\lambda}''(x)-\varphi_r''(x)\varphi_{r+\lambda}'(x)) \\
- & L^{r+\lambda}_{\varphi}(x)(g_d''(x)\varphi_{r+\lambda}'(x)-g_d'(x)\varphi_{r+\lambda}''(x)).
\end{align*}
Using Lemma \ref{Lemma:LinksBetweenFunctions} and resolvent equation (\ref{Eq:ResolventEquation}) we see that 
$$
g_d = -(R_r \theta_d)=-\lambda (R_{r+\lambda}(-g_d+\lambda^{-1}\theta_d)).
$$ 
Therefore, by the lemma \ref{Lemma:ResolventRepresentationsLandK}, we arrive at
$$
\frac{1}{2}\lambda \sigma^2(x)S'(x)L^{r+\lambda}_{-g_d+\lambda^{-1}\theta}(x)= g_d''(x)\varphi_{r+\lambda}'(x)-g_d'(x)\varphi_{r+\lambda}''(x)
$$
and also
$$
\frac{1}{2}\lambda \sigma^2(x)S'(x)L^{r+\lambda}_{\varphi} = \varphi_r'(x)\varphi_{r+\lambda}''(x)-\varphi_r''(x) \varphi_{r+\lambda}'(x).
$$
Hence, the derivative reads as
\begin{align*}
& \varphi_{r+\lambda}'(x)^2 Q'(g_d, \varphi_r; x) \\
& = \frac{1}{2}\lambda \sigma^2(x)S'(x)L^{r+\lambda}_{g_d}(x)L^{r+\lambda}_{\varphi}(x)+\frac{1}{2}\lambda \sigma^2(x)S'(x)L^{r+\lambda}_{-g_d+\lambda^{-1}\theta_d}(x)L^{r+\lambda}_{\varphi}(x) \\
& = \frac{1}{2}\lambda \sigma^2(x)S'(x)L^{r+\lambda}_{\varphi}(x)(L^{r+\lambda}_{g_d}(x)+L^{r+\lambda}_{-g_d+\lambda^{-1}\theta_d}(x)) \\
& = \frac{1}{2} \sigma^2(x)S'(x)L^{r+\lambda}_{\varphi}(x)L^{r+\lambda}_{\theta_d}(x).
\end{align*}
This implies that the monotonicity is determined only by the sign of $L^{r+\lambda}_{\theta_d}$ as the sign of all the other terms are known: all the other terms are positive, except $L^{r+\lambda}_{\varphi}$ is negative. Similarly, we can calculate that
\begin{align*}
& \psi_{r+\lambda}'(x)^2 H'(\varphi_r, g_u; x) = \frac{1}{2} \sigma^2(x) S'(x)K^{r+\lambda}_{\varphi}(x)K^{r+\lambda}_{\theta_u}(x), \\
& \varphi_{r+\lambda}'(x)^2 Q'(g_d, \psi_r; x) = \frac{1}{2} \sigma^2(x)S'(x)L^{r+\lambda}_{\psi}(x)L^{r+\lambda}_{\theta_d}(x), \\
& \psi_{r+\lambda}'(x)^2 H'(\psi_r, g_u; x) = \frac{1}{2} \sigma^2(x) S'(x)K^{r+\lambda}_{\psi}(x)K^{r+\lambda}_{\theta_u}(x), 
\end{align*}
where $L^{r+\lambda}_{\psi}>0$, $K^{r+\lambda}_{\varphi} > 0$ and $K^{r+\lambda}_{\psi}<0$. This proves the claim.
\end{proof}

\begin{remark} \label{Remark:MonotonicityOfHandQ} Lemma \ref{Lemma:LandKsigns} and \ref{Lemma:MonotonicityOfHandQ} together show that there exists two points $\hat{x} < x_d^*$ and $\tilde{x}>x_u^*$ such that 
\begin{enumerate}[(i)]
\item  $Q(  g_d, \psi_r; x)$ is increasing on $(0,\hat{x})$ and decreasing on $(\hat{x},\infty)$, %OY
\item  $H( \psi_r, g_u; x)$ is increasing on $(0,\tilde{x})$ and decreasing on $(\tilde{x},\infty)$, %VY
\item  $Q( g_d, \varphi_r; x)$ is decreasing on $(0,\hat{x})$ and increasing on $(\hat{x},\infty)$, %OA
\item  $H( \varphi_r, g_u; x)$ is decreasing on $(0,\tilde{x})$ and increasing on $(\tilde{x},\infty)$. %VA
\end{enumerate} 
\end{remark}

The following lemma about the minimal excessive functions $\psi_r$ and $\varphi_r$ is useful in further analysis of $Q$ and $H$.
\begin{lemma} \label{Lemma:PhiPsiOrdering} The minimal excessive functions $\psi_{r}(x)$ and $\varphi_{r}(x)$ satisfy the following inequalities for $z < x < y$
\begin{align*}
    \frac{\psi_{r+\lambda}(z)}{\psi_{r}(z)} \leq \frac{\psi_{r+\lambda}(x)}{\psi_{r}(x)} \leq  \frac{\psi_{r+\lambda}'(x)}{\psi_{r}'(x)}, \\
    \frac{\varphi_{r+\lambda}(z)}{\varphi_{r}(z)} \leq  \frac{\varphi_{r+\lambda}(x)}{\varphi_{r}(x)} \leq \frac{\varphi_{r+\lambda}'(x)}{\varphi_{r}'(x)} 
\end{align*}
\end{lemma}
\begin{proof}
We observe that for all $s > 0$ and $z \leq x$ we have (see \cite{BorodinSalminen} pp. 18)
\begin{equation*}
\mathbb{E}_z[e^{-s\tau_x} \mid \tau_x < \infty] = \dfrac{\psi_{s}(z)}{\psi_{s}(x)},
\end{equation*}
where  $\tau_x = \inf\{ t \geq 0 \mid X_t = x \}$. Thus, 
\begin{equation} \label{MEFL1}
   \dfrac{\psi_{r}(z)}{\psi_{r}(x)} \geq \dfrac{\psi_{r+\lambda}(z)}{\psi_{r+\lambda}(x)}.
\end{equation}
Moreover, utilizing that $(\mathcal{A}-r)\psi_{r+\lambda}= (\mathcal{A}-(r+\lambda))\psi_{r+\lambda} + \lambda\psi_{r+\lambda} = \lambda\psi_{r+\lambda}$ with the Corollary 3.2 of \cite{Alvarez2004}, we have
\begin{equation*}
    \psi_{r}'(x) \psi_{r+\lambda}(x)  - \psi_{r+\lambda}'(x) \psi_{r}(x) = - \lambda S'(x)  \int_{0}^{x} \psi_{r}(y)\psi_{r+\lambda}(y)m'(y) \leq 0.
\end{equation*}
Reorganizing the above we get
\begin{equation} \label{MEFL2}
    \frac{\psi_{r}'(x)}{\psi_{r}(x)} \leq \frac{\psi_{r+\lambda}'(x) }{\psi_{r+\lambda}(x)}. 
\end{equation}
Combining (\ref{MEFL1}) and (\ref{MEFL2}) yields inequalities for $\psi_r$ and the inequalities for $\varphi_r$ are proven similarly.
\end{proof}

\section{The solution}
\subsection{The associated free boundary problem}
Similar to \cite{DupuisWang2002, Lempa2014, LempaSaarinen2021}, we can use straightforward heuristic arguments to formulate a free boundary problem for the candidate value function of the control problem. Denoting the positive twice continuously differentiable candidate value function by $F$ and two constant boundaries by $a$ and $b$, where $a <b$, the heuristics give us the free boundary problem
\begin{align}
& (\mathcal{A}-r)F(x) = -\pi(x), &  a < x < b & \label{Bb:1}\\  
& (\mathcal{A}-(r+\lambda))F(x) = -\pi(x)-\lambda(\gamma_d(x-b)+F(b)), & x \geq b & \label{Bb:2}\\ 
& (\mathcal{A}-(r+\lambda))F(x) = -\pi(x)-\lambda(\gamma_u(x-a)+F(a)), & x \leq a & \label{Bb:3}\\  
& F'(b) = \gamma_d, &  & \label{Bb:4}\\
& F'(a) = \gamma_u. &  & \label{Bb:5}
\end{align}
%\begin{equation*}
%        (\mathcal{A}-r)F(x)+\pi(x)= 
%        \begin{cases}
%            \lambda((F(x)-\gamma_d x)-(F(b)-\gamma_d b))  \\
%            0 \\
%            \lambda((F(x)-\gamma_u x)-(F(a)-\gamma_u a))
%        \end{cases}
%\end{equation*}
We first consider the equation (\ref{Bb:1}). This equation gives
\begin{equation*}
F(x) = B_1 \varphi_r(x) + B_2 \psi_r(x) + (R_r \pi)(x), \quad a<x<b,
\end{equation*}
where $B_1$ and $B_2$ are constants. Using the first order conditions (\ref{Bb:4}) and (\ref{Bb:5}), we find after some algebraic manipulation that
\begin{equation*}
B_1 = \frac{\gamma_d \psi_r'(a)-\gamma_u \psi_r'(b)+\psi_r'(b)(R_r \pi)'(a)- \psi_r'(a)(R_r \pi)'(b)}{\varphi_r'(b)\psi_r'(a)-\varphi_r'(a)\psi_r'(b)}
\end{equation*}
\begin{equation*}
B_2 =  \frac{-\gamma_d \varphi_r'(a) + \gamma_u \varphi_r'(b) - \varphi_r'(b)(R_r \pi)'(a) + \varphi_r'(a)(R_r \pi)'(b)}{\varphi_r'(b)\psi_r'(a)-\varphi_r'(a)\psi_r'(b) }.
\end{equation*}
Because the boundaries of the state space are natural, we find from the differential equations (\ref{Bb:2}) and (\ref{Bb:3}) that particular solutions to these are
\begin{align*}
& (R_{r+\lambda} \pi_{\gamma_d})(x)+\frac{\lambda}{\lambda + r}(F(b)-\gamma_d b), \quad x>b \\
& (R_{r+\lambda} \pi_{\gamma_u})(x)+\frac{\lambda}{\lambda + r}(F(a)-\gamma_u a), \quad x<a.
\end{align*}
Thus, by growth conditions given by lemma \ref{Lemma:ValueFunctioBounds}, continuity over the boundaries and first order conditions (\ref{Bb:4}) and (\ref{Bb:5}) we see that the general solutions can be written as
\begin{align*}
F(x) & =  C \varphi_{r+\lambda}(x)+ (R_{r+\lambda} \pi_{\gamma_d})(x) \\ &+ \frac{\lambda}{r}\bigg[ C \varphi_{r+\lambda}(b)+ (R_{r+\lambda} \pi_{\gamma_d})(b)-\gamma_d b \bigg], \quad \, x>b, \nonumber \\
F(x) & = D \psi_{r+\lambda}(x)+ (R_{r+\lambda} \pi_{\gamma_u})(x) \\ &+ \frac{\lambda}{r}\bigg[D \psi_{r+\lambda}(a)+ (R_{r+\lambda} \pi_{\gamma_u})(a)-\gamma_u a \bigg], \quad \, x<a, \nonumber
\end{align*}
where the functions $\pi_{\gamma_d}$, $\pi_{\gamma_u}$ are defined in (\ref{Eq:PigammaDefinition}) and the constants $C$ and $D$ are
\begin{align*}
& C = \frac{\gamma_d - (R_{r+\lambda}\pi_{\gamma_d})'(b)}{\varphi_{r+\lambda}'(b)}, \\
& D = \frac{\gamma_u - (R_{r+\lambda}\pi_{\gamma_u})'(a)}{\psi_{r+\lambda}'(a)}.
\end{align*}
To solve the boundary points $a$ and $b$, we first use the $C^2$ condition at the upper boundary point $b$. This reads as
\begin{equation} \label{Eq:UpperboundaryC2Condition}
B_1 \varphi_r''(b) + B_2 \psi_r''(b) + (R_r \pi)''(b) = C \varphi_{r+\lambda}''(b) + (R_{r+ \lambda }\pi_{\gamma_d})''(b).
\end{equation}
To simplify this, we recall the definition of the functions $g_n$ (\ref{Eq:gnDefinition}) and lemma \ref{Lemma:LinksBetweenFunctions}.
These yield for the right-hand side of (\ref{Eq:UpperboundaryC2Condition})
\begin{align*}
& \frac{\gamma_d - (R_{r+\lambda}\pi_{\gamma_d})'(b)}{\varphi_{r+\lambda}'(b)} \varphi_{r+\lambda}''(b) + (R_{r+ \lambda }\pi_{\gamma_d})''(b) \\
= & \frac{g'_d(b) \varphi''_{r+\lambda}(b)-\lambda(R_{r+\lambda}g_d)'(b)\varphi''_{r+\lambda}(b)}{\varphi_{r+\lambda}'(b)} + \lambda(R_{r+\lambda}g_d)''(b) + (R_r \pi)''(b) \\
= & g'_d(b) \frac{ \varphi''_{r+\lambda}(b)}{\varphi_{r+\lambda}'(b)} +(R_r \pi)''(b) \\ & +  \frac{\lambda(R_{r+\lambda}g_d)''(b)\varphi'_{r+\lambda}(b)-\lambda(R_{r+\lambda}g_d)'(b)\varphi''_{r+\lambda}(b)}{\varphi_{r+\lambda}'(b)}.
\end{align*}
Here, by lemma \ref{Lemma:ResolventRepresentationsLandK}, the second term is an integral operator 
$$
\frac{-2 \lambda S'(b)}{\sigma^{2}(b)} L_{g_d}(b).
$$
To deal with the left-hand side of (\ref{Eq:UpperboundaryC2Condition}), a straight forward algebraic manipulation yields
\begin{align*}
& \frac{\gamma_d \psi_r'(a)-\gamma_u \psi_r'(b)+\psi_r'(b)(R_r \pi)'(a)- \psi_r'(a)(R_r \pi)'(b)}{\varphi_r'(b)\psi_r'(a)-\varphi_r'(a)\psi_r'(b)} \varphi_r''(b) \\
+ &\frac{-\gamma_d \varphi_r'(a) + \gamma_u \varphi_r'(b) - \varphi_r'(b)(R_r \pi)'(a)+\varphi_r'(a)(R_r \pi)'(b)}{\varphi_r'(b)\psi_r'(a)-\varphi_r'(a)\psi_r'(b) } \psi_r''(b) \\
= & \frac{g_d'(b)(\varphi''_r(b)\psi_r'(a)-\varphi_r'(a)\psi''_r(b))+g_u'(a)(\psi''_r(b)\varphi_r'(b)-\varphi_r''(b)\psi_r'(b))}{\varphi_r'(b)\psi_r'(a)-\varphi_r'(a)\psi_r'(b)}. 
\end{align*}
We note that $-2r S'(b)\sigma^{-2}(b) B_r=\psi''_r(b)\varphi_r'(b)-\varphi_r''(b)\psi_r'(b)$. Next we combine two terms from above, one from left side and one from right side of (\ref{Eq:UpperboundaryC2Condition}), yielding
\begin{align*}
& \frac{g_d'(b)(\varphi''_r(b)\psi_r'(a)-\varphi_r'(a)\psi''_r(b))}{\varphi_r'(b)\psi_r'(a)-\varphi_r'(a)\psi_r'(b)} - g'_d(b) \frac{ \varphi''_{r+\lambda}(b)}{\varphi_{r+\lambda}'(b)} \\
= & \frac{g_d'(b)(\varphi_{r+\lambda}'(b)(\varphi''_r(b)\psi_r'(a)-\varphi_r'(a)\psi''_r(b))}{\varphi_{r+\lambda}'(b)(\varphi_r'(b)\psi_r'(a)-\varphi_r'(a)\psi_r'(b))} \\ & - \frac{ \varphi''_{r+\lambda}(b)(\varphi_r'(b)\psi_r'(a)-\varphi_r'(a)\psi_r'(b)))}{\varphi_{r+\lambda}'(b)(\varphi_r'(b)\psi_r'(a)-\varphi_r'(a)\psi_r'(b))} \\
= & \frac{g_b'(b)\big(\psi_r'(a)(\varphi_{r+\lambda}'(b)\varphi''_r(b)-\varphi''_{r+\lambda}(b)\varphi_r'(b))}{\varphi_{r+\lambda}'(b)(\varphi_r'(b)\psi_r'(a)-\varphi_r'(a)\psi_r'(b))} \\ & + \frac{\varphi_r'(a)( \varphi''_{r+\lambda}(b)\psi_r'(b)-\varphi'_{r+\lambda}(b)\psi_r''(b)) \big)}{\varphi_{r+\lambda}'(b)(\varphi_r'(b)\psi_r'(a)-\varphi_r'(a)\psi_r'(b))}.
\end{align*}
Again by lemma \ref{Lemma:ResolventRepresentationsLandK}, we find that
$$
\varphi_{r+\lambda}'(b)\varphi''_r(b)-\varphi''_{r+\lambda}(b)\varphi_r'(b) = -2\lambda S'(b)\sigma^{-2}(b) L_{\varphi}(b)
$$
and 
$$
\varphi''_{r+\lambda}(b)\psi_r'(b)-\varphi'_{r+\lambda}(b)\psi_r''(b) =2 \lambda S'(b)\sigma^{-2}(b) L_{\psi}(b).
$$
Finally, using the above calculations the $C^2$ condition (\ref{Eq:UpperboundaryC2Condition}) reads as
\begin{equation*}
\frac{g_d'(b)\big(
\varphi_r'(a)L_{\psi}(b) -\psi_r'(a)L_{\varphi}(b) \big)-\frac{r}{\lambda}g_u'(a)\varphi_{r+\lambda}'(b)B_r}{\varphi_r'(b)\psi_r'(a)-\varphi_r'(a)\psi_r'(b)}  + L_{g_d}(b) = 0.
\end{equation*}
Reorganizing the above we have
\begin{equation*}
\scalebox{0.96}{$
\psi_r'(a) \bigg( \frac{\varphi_r'(b) L_{g_d}(b)-g_d'(b)L_{\varphi}(b)}{\varphi_{r+\lambda}'(b)}\bigg)+\varphi_r'(a) \bigg( \frac{g_d'(b)L_{\psi}(b)-\psi_r'(b)L_{g_d}(b)}{\varphi_{r+\lambda}'(b)}\bigg) = \frac{r}{\lambda}B_r g_u'(a)$}.
\end{equation*}
By entirely similar arguments we arrive at the lower boundary to the equation
\begin{equation*}
\scalebox{0.96}{$
\psi_r'(b) \bigg( \frac{g_u'(a)K_{\varphi}(a)-\varphi_r'(a) K_{g_u}(a)}{\psi_{r+\lambda}'(a)}\bigg)+\varphi_r'(b)\bigg( \frac{\psi_r'(a)K_{g_u}(a)-g_u'(a)K_{\psi}(a)}{\psi_{r+\lambda}'(a)}\bigg) = \frac{r}{\lambda}B_r g_d'(b)$}.
\end{equation*}
Since the function $F$ is $r$-harmonic in the interval $(a,b)$, we find by uniqueness, that the optimal thresholds should satisfy the pair of equations
\begin{align*}
& \frac{\psi_r'(a)K_{g_u}(a)-g_u'(a)K_{\psi}(a)}{\psi'_{r+\lambda}(a)} = \frac{g_d'(b)L_{\psi}(b)-\psi_r'(b)L_{g_d}(b)}{\varphi'_{r+\lambda}(b)}, \\
& \frac{\varphi_r'(a)K_{g_u}(a)-g_u'(a) K_{\varphi}(a)}{\psi'_{r+\lambda}(a)} = \frac{g_d'(b) L_{\varphi}(b)-\varphi_r'(b)L_{g_d}(b)}{\varphi'_{r+\lambda}(b)}.
\end{align*}
Using the notation defined in lemma \ref{Lemma:MonotonicityOfHandQ} the pair can be written as
\begin{equation} \label{Eq:TheOptimalPair}
\begin{cases} 
H(\psi_r, g_u; a) = Q( g_d, \psi_r; b),  \\
H(\varphi_r, g_u; a) = Q(g_d, \varphi_r; b).
\end{cases}
\end{equation}

\subsection{Uniqueness and existence}

In this section we prove that an unique solution $(a^*, b^*)$ to the pair of equations (\ref{Eq:TheOptimalPair}) exists. To this end, we use a slight generalisation of the method that was first used in \cite{Lempa2010} and \cite{AlvarezLempa2008} in a analogous setting and notice that in the solution $(a^*, b^*)$ the point $a^*$ must necessarily be a fixed point of the function
\begin{equation*} 
K(x) = H_{u, \varphi}^{-1}( Q_{d, \varphi}( Q_{d,\psi}^{-1}( H_{u, \psi}( x)))),
\end{equation*}
where slightly shorter notation 
\begin{align*}
& H_{u, \varphi}(x) = H(\varphi_r,g_u; x), \quad
H_{u, \psi}(x) = H(\psi_r, g_u; x), \\
& Q_{d,\psi}(x) = Q(g_d, \psi_r;x), \quad
Q_{d,\varphi}(x) = Q(g_d,\varphi_r;x),
\end{align*}
is introduced. Thus, in order to prove that the solution exists we must first ensure that $K$ is well-defined and then study its fixed points.

The following is the main result on the uniqueness of the solution.

\begin{proposition} \label{Prop:Uniqueness}
Let the assumptions 1 and 2 hold and assume that the pair of equations
\begin{equation*}
\begin{cases}
H( \psi_r, g_u; a) = Q(  g_d, \psi_r; b),  \\
H( \varphi_r, g_u; a) = Q( g_d, \varphi_r; b).
\end{cases}
\end{equation*}
has a solution $(a^*,b^*)$. Then the solution is unique. 
\end{proposition}
\begin{proof}
Let $K:(0, \tilde{x}] \to (0, \tilde{x}]$ be defined as
\begin{equation} \label{Eq:FunctionKDefinition}
K(x) = \check{H}_{u, \varphi}^{-1}( \hat{Q}_{d, \varphi}( \hat{Q}_{d,\psi}^{-1}( \check{H}_{u, \psi}( x)))),
\end{equation}
where $\hat{\cdot}$ and $\check{\cdot}$ are restrictions to domains  $[\hat{x}, \infty)$ and $(0, \tilde{x}]$ respectively.  We notice that if a solution $(a^*,b^*)$ to the pair exists, then $a^*$ must be fixed point of $K$. Because the functions $H$ and $Q$ are monotonic in their domains we get
\begin{equation} \label{Eq:Kincreasing}
\begin{aligned} 
    K'(x)  = & \check{H}_{u, \varphi}^{-1'}(\hat{Q}_{d, \varphi}(\hat{Q}_{d, \psi}^{-1}(\check{H}_{u, \psi}(x)))) \cdot \hat{Q}_{d, \varphi}'(\hat{Q}_{d, \psi}^{-1}(\check{H}_{u, \psi}(x))) \\ 
    \cdot & \hat{Q}_{d, \psi}^{-1'}(\check{H}_{u, \psi}(x)) \cdot \check{H}_{u, \psi}'(x) > 0,
\end{aligned}
\end{equation}
and hence $K$ is increasing in its domain $(0,\tilde{x}_i]$. We observe using the fixed point property that
\begin{equation*}
    K'(a^*) = \frac{\hat{Q}_{d, \varphi}'(b^*)}{\hat{Q}_{d, \psi}'(b^*)}  \frac{\check{H}_{u, \psi}'(a^*)}{\check{H}_{u, \varphi}'(a^*)} =  \frac{L_{\varphi}^{r+\lambda}(b^*)}{L_{\psi}^{r+\lambda}(b^*)}  \frac{K_{\varphi}^{r+\lambda}(a^*)}{K_{\psi}^{r+\lambda}(a^*)} < \frac{\varphi_r(b^*)}{\varphi_r(a^*)} \frac{\psi_r(a^*)}{\psi_r(b^*)} < 1.
\end{equation*}
Consequently, whenever $K$ intersects the diagonal of $\mathbb{R}_+$, the intersection must be from above.
\end{proof}

To prove the existence of the solution we need to study the function $K$ in more detail. First, we analyze the limiting properties of the functions $Q$ and $H$. We notice that by lemma \ref{Lemma:LinksBetweenFunctions} and \ref{Lemma:ResolventRepresentationsLandK} we have
\begin{align*}
   & \varphi_{r+\lambda}'(x) Q( g_d, \varphi_r; x) \\
   & = g_d'(x)L_{\varphi}^{r+\lambda}(x)-\varphi_r'(x)L_{g_d}^{r+\lambda}(x) \\
   & = -\frac{\sigma(x)^2}{2 \lambda S'(x)} \big[ \varphi_r(x)(\varphi_{r+\lambda}''(x) (R_{r+\lambda}\theta_d)'(x) - \varphi_{r+\lambda}'(x) (R_{r+\lambda} \theta_d)''(x)) \\
   & + \varphi_{r+\lambda}'(x)(\varphi_r''(x)g_d'-\varphi_r'(x)g_d''(x)) \big] \\
   & = \frac{1}{\lambda}\big[\varphi_{r+\lambda}'(x)L_{\theta_d}^r(x)- \varphi_r'(x)L_{\theta_d}^{r+\lambda}(x)\big].
\end{align*}
Carrying out similar calculations for $H( \psi_r, g_u; a),$ $Q(  g_d, \psi_r; b)$ and $H( \varphi_r, g_u; a)$ the pair of equations reads as
\begin{equation} \label{Eq:PairRepresented}
\begin{aligned} 
    &  \frac{\psi_r'(a)}{\psi_{r+\lambda}'(a)} K_{\theta_u}^{r+\lambda}(a) - K_{\theta_u}^r(a) = -\frac{\psi_r'(b)}{\varphi_{r+\lambda}'(b)} L_{\theta_d}^{r+\lambda}(b)- K_{\theta_d}^r(b), \\
    & \frac{\varphi_r'(a)}{\psi_{r+\lambda}'(a)} K_{\theta_u}^{r+\lambda}(a) +  L_{\theta_u}^r(a) = -\frac{\varphi_r'(b)}{\varphi_{r+\lambda}'(b)}L_{\theta_d}^{r+\lambda}(b) + L_{\theta_d}^r(b).
\end{aligned}
\end{equation}
Assuming that $x > \max(\tilde{x},x_d^0)$ (here $x_d^0$ is the unique root of $\theta_d(x)$) we find using mean value theorem that 
\begin{align*}
& Q(g_d, \psi_r; x) = -\frac{\psi_r'(x)}{\varphi_{r+\lambda}'(x)} L_{\theta_d}^{r+\lambda}(x)- K_{\theta_d}^r(x)  \\
& = -(r+\lambda)\frac{\psi_r'(x)}{\varphi_{r+\lambda}'(x)} \int_x^{\infty} \theta_d(y) \varphi_{r+\lambda}(y)m'(y)dy - r\int_0^{x} \theta_d(y)\psi_r(y)m'(y)dy \\
& = \frac{\psi_r'(x)}{S'(x)} \theta_d(\xi_x) - r \int_{0}^x \theta_d(y)\psi_r(y)m'(y)dy \\
& < \frac{\psi_r'(x)}{S'(x)}(\theta_d(\xi_x)- \theta_d(x)) < 0,
\end{align*}
where $\xi_x \in (x, \infty)$. Hence, $\lim_{x \to \infty} Q(g_d, \psi_r; x) < 0$. Similarly, we find that $\lim_{x \to 0} H(\varphi_r, g_u; x) > 0$.

Moreover, if $x < \min(\hat{x}, x_u^0)$ (here $x_u^0$ is the unique root of $\theta_u(x)$) we get by mean value theorem and lemma \ref{Lemma:PhiPsiOrdering} that
\begin{align*}
    & H(\psi_r, g_u; x) \\
    & =  \frac{\psi_r'(x)}{\psi_{r+\lambda}'(x)} K_{\theta_u}^{r+\lambda}(x) - K_{\theta_u}^r(x) \\
    & = (r+\lambda)\frac{\psi_r'(x)}{\psi_{r+\lambda}'(x)}\int_0^{x} \theta_u(y)\frac{\psi_{r+\lambda}(y)}{\psi_r(y)}\psi_{r}(y)m'(y)dy - r\int_0^{x} \theta_u(y)\psi_r(y)m'(y)dy \\
    & = (r+\lambda)\underbrace{\frac{\frac{\psi_r'(x)}{\psi_{r+\lambda}'(x)}}{\frac{\psi_r(\xi_x)}{\psi_{r+\lambda}(\xi_x)}}}_{<1}\int_0^{x} \theta_u(y)\psi_{r}(y)m'(y)dy - r\int_0^{x} \theta_u(y)\psi_r(y)m'(y)dy,
\end{align*}
and thus by continuity $\lim_{x \to 0+}H(\psi_r, g_u;x) = 0$. Similarly, we find that $\lim_{x \to \infty} Q(g_d,\varphi_r;x)=0$. Also, because $\hat{x}$ is a root of $L_{\theta_d}^{r+\lambda}$ and $\hat{x} < x^*_d$ (consequently $\theta$ is increasing on $(0, \hat{x})$) we have that
\begin{align*}
    Q(g_d, \psi_r; \hat{x}) = -r\int_0^{\hat{x}} (\theta_d(y)-\theta_d(\hat{x}))\psi_r(y)m'(y)dy >0.
\end{align*}
And again, similarly $H(\varphi_r,g_u;\tilde{x})<0$. We summarize these findings:
\begin{equation} \label{Eq:LimitsOfHandQ}
\begin{cases}
     \,\, Q(g_d, \varphi_r; \infty) = 0, \\
     \,\, Q(g_d, \psi_r; \infty) < 0, \\
    \,\, Q(g_d, \psi_r; \hat{x}) > 0, 
\end{cases}
\begin{cases}
     \,\, H(\psi_r, g_u; 0+) = 0 \\
     \,\, H(\varphi_r, g_u; 0+) > 0, \\
     \,\, H(\varphi_r, g_u; \tilde{x}) < 0,
\end{cases}
\end{equation}

Unfortunately, the analysis so far is not enough for the existence of the solution. To guarantee that the fixed point of the function $K$ exists, the inequalities
\begin{align*}
    & H( \psi_r, g_u; \tilde{x}) \leq Q(  g_d, \psi_r; \hat{x}),  \\
    & H( \varphi_r, g_u; \tilde{x}) \leq Q( g_d, \varphi_r; \hat{x})
\end{align*}
have to hold. Our assumptions are not enough for these inequalities to hold in general, but the next lemma gives an easily verifiable sufficient condition.
\begin{lemma} \label{Lemma:HQorder}
Assume that $\tilde{x} < \hat{x}$. Then 
\begin{equation*}
\begin{aligned}
    & H( \psi_r, g_u; \tilde{x}) \leq Q(  g_d, \psi_r; \hat{x}),  \\
    & H( \varphi_r, g_u; \tilde{x}) \leq Q( g_d, \varphi_r; \hat{x})
\end{aligned}
\end{equation*}
\end{lemma}
\begin{proof}
We prove the first inequality as the second is shown similarly. Because $\tilde{x}$ and $\hat{x}$ are zeros of $K_{\theta_u}^{r+\lambda}$ and $L_{\theta_d}^{r+\lambda}$ respectively, we note by (\ref{Eq:PairRepresented}) that the first inequality is equivalent to
\begin{equation*}
    K_{\theta_d}^r(\hat{x})-K_{\theta_u}^r(\tilde{x}) \leq 0.
\end{equation*}
By lemma \ref{Lemma:LandKsigns} we have that $\hat{x} < x^*_d$ and $\tilde{x} > x_u^*$. Thus, by our assumption $\tilde{x} < \hat{x}$ we have that $x_u^* < \tilde{x} < \hat{x} < x_d^*$. Hence, we find that
\begin{align*}
    & K_{\theta_d}^r(\hat{x})-K_{\theta_u}^r(\tilde{x}) \\
    & = \int_0^{\hat{x}} \theta_d(z) \psi_r(z)m'(z)dz - \theta_d(\hat{x}) \int_0^{\hat{x}} \psi_r(z)m'(z)dz \\ &-\int_0^{\tilde{x}}\theta_u(z)\psi_r(z)m'(z)dz + \theta_u(\tilde{x})\int_0^{\tilde{x}} \psi_r(z)m'(z)dz \\
    & \leq \int_0^{\tilde{x}} (\theta_d(z)-\theta_d(\tilde{x})) \psi_r(z)m'(z)dz + \int_{\tilde{x}}^{\hat{x}} (\theta_d(z)- \theta_d(\hat{x})) \psi_r(z)m'(z)dz \\
    &-\int_0^{\tilde{x}}(\theta_u(z)-\theta_u(\tilde{x}))\psi_r(z)m'(z)dz, \\
    & \leq \int_0^{\tilde{x}} (\theta_d(z)-\theta_d(\tilde{x})-\theta_u(z)+\theta_u(\tilde{x})) \psi_r(z)m'(z)dz \leq 0,
\end{align*}
where the first two inequalities follow from part $(v)$ of assumptions \ref{Ass:MainAssumptions} and last from parts $(i)$ and $(iv)$ of assumptions \ref{Ass:MainAssumptions}. 
\end{proof}

The assumption $\tilde{x} < \hat{x}$ may seem restricting, but based on numerical calculations it cannot easily be relaxed, because there exists cases under the main assumptions \ref{Ass:MainAssumptions}, where $\tilde{x} > \hat{x}$ and the solution does not exist. Furthermore, because the points $\tilde{x}$ and $\hat{x}$ are known to be the unique roots of the functionals $K_{\theta_u}^{r+\lambda}$ and $L_{\theta_d}^{r+\lambda}$, it is straightforward to calculate them, at least numerically, and verify the assumption.

\begin{proposition} \label{Prop:Existence}
Let the assumptions 1 and 2 hold. Assume further that $\tilde{x} < \hat{x}$, then the pair of equations
\begin{equation*}
\begin{cases}
H( \psi_r, g_u; a) = Q(  g_d, \psi_r; b),  \\
H( \varphi_r, g_u; a) = Q( g_d, \varphi_r; b).
\end{cases}
\end{equation*}
has a unique solution $(a^*, b^*)$. 
\end{proposition}
\begin{proof}
Define the function $K:(0, \tilde{x}] \to (0, \tilde{x}]$ as in (\ref{Eq:FunctionKDefinition}). Then by lemma \ref{Lemma:HQorder}, remark \ref{Remark:MonotonicityOfHandQ} and limiting properties (\ref{Eq:LimitsOfHandQ}) the function $K$ is well-defined. Further, by (\ref{Eq:Kincreasing}) $K$ is monotonic mapping. Thus, $K$ is a monotonic mapping from a set to its open subset, and hence it must have at least one fixed point, which we denote by $a^*$.  Then the pair $(a^*, b^*)$, where $b^* = Q_{d, \psi}^{-1}(H_{u,\psi}(a^*))$, is a solution to the equations (\ref{Eq:TheOptimalPair}). The uniqueness follow from proposition \ref{Prop:Uniqueness}. 
\end{proof}

\subsection{Verification}

We begin by stating the verification theorem.
\begin{proposition} \label{Theorem:Verification} Let assumptions \ref{Ass:MainAssumptions} hold and denote by $(a^*,b^*)$ the unique solution to the necessary conditions
\begin{equation*}
\begin{cases}
H( \psi_r, g_u; a) = Q(  g_d, \psi_r; b),  \\
H( \varphi_r, g_u; a) = Q( g_d, \varphi_r; b).
\end{cases}
\end{equation*}
Then the optimal policy is as follows. If the controlled process $X^{\zeta}$ is not inside the interval $(a^*, b^*)$ at a jump time $T_i$ of $N$, i.e. $X^{\zeta}_{T_{i-}} \not \in (a^*, b^*)$ for any $i$, the optimal policy is to take the controlled process $X^{\zeta}$ to the closest boundary of the interval $(a^*, b^*)$. Moreover, the optimal value function $V(x)$ reads as
\begin{equation} \label{Eq:CandidateValue}
V(x)=
    \begin{cases}
     \frac{\gamma_d - (R_{r+\lambda}\pi_{\gamma_d})'(b)}{\varphi_{r+\lambda}'(b)} \varphi_{r+\lambda}(x)+ (R_{r+\lambda} \pi_{\gamma_d})(x) + A_{d}(b), \\
    B_1 \varphi_r(x) + B_2 \psi_r(x) + (R_r \pi)(x), \\
    \frac{\gamma_u - (R_{r+\lambda}\pi_{\gamma_u})'(a)}{\psi_{r+\lambda}'(a)} \psi_{r+\lambda}(x)+ (R_{r+\lambda} \pi_{\gamma_u})(x)+ A_{u}(a),
    \end{cases}
\end{equation}
where
\begin{align*}
& B_1 = \frac{\gamma_d \psi_r'(a)-\gamma_u \psi_r'(b)+\psi_r'(b)(R_r \pi)'(a)- \psi_r'(a)(R_r \pi)'(b)}{\varphi_r'(b)\psi_r'(a)-\varphi_r'(a)\psi_r'(b)}, \\
& B_2 =  \frac{-\gamma_d \varphi_r'(a) + \gamma_u \varphi_r'(b) - \varphi_r'(b)(R_r \pi)'(a) + \varphi_r'(a)(R_r \pi)'(b)}{\varphi_r'(b)\psi_r'(a)-\varphi_r'(a)\psi_r'(b) }, \\
& A_{d}(b)=\frac{\lambda}{r}\bigg[ C \varphi_{r+\lambda}(b)+ (R_{r+\lambda} \pi_{\gamma_d})(b)-\gamma_d b \bigg], \\
& A_{u}(a)= \frac{\lambda}{r}\bigg[D \psi_{r+\lambda}(a)+ (R_{r+\lambda} \pi_{\gamma_u})(a)-\gamma_u a \bigg].
\end{align*}
\end{proposition}

Before proving the proposition, we show few properties of the candidate value function $F$.

%$$ F''(b) = -\frac{2S'(b)}{\sigma^2(b) \varphi_{r+\lambda}'(b)} L^{r+\lambda}_{\theta_d}(b) $$ and because $b > %\hat{x}$ we have that $F''(b) < 0$. Similarly we get $F''(a) > 0$. 
\begin{lemma} \label{Lemma:CandidateValueProperties}
\hspace{1cm}
\begin{enumerate}[($i$)]
    \item $F''(x) \leq 0$ for all $x \in (a,b)$ 
    \item The function $x \mapsto F(x)-\gamma_d x$ has an unique global maximum at $b$. Similarly, the function $x \mapsto F(x)-\gamma_u x$ has an unique global maximum at $a$.
\end{enumerate}
\end{lemma}
\begin{proof}
The item $(i)$ is same as part (A) of lemma 4.3 in \cite{Matomaki2012}.

To prove the item $(ii)$, we find by straight differentiation in (\ref{Eq:CandidateValue}) that when $x > b$ 
\begin{equation*}
    F'(x) -\gamma_d= \varphi_{r+\lambda}'(x) \bigg[ \frac{(R_{r+\lambda}\pi_{\gamma_d})'(x)-\gamma_d}{\varphi_{r+\lambda}'(x)} - \frac{(R_{r+\lambda}\pi_{\gamma_d})'(b)-\gamma_d}{\varphi_{r+\lambda}'(b)}\bigg] < 0,
\end{equation*}
where we have used that $b > \hat{x} $ and lemma \ref{Lemma:LandKsigns}. Similarly, when $x < a$ we find that $F'(x) -\gamma_u > 0$. Furthermore, as $F'(a+) = \gamma_u > \gamma_d = F'(b-)$ we find by item $(i)$ that we must have $\gamma_u \leq F'(x) \leq \gamma_d$, when $x \in (a,b)$. Hence, the item $(i)$ follows by $F'(x) > \gamma_u >  \gamma_d$ when $x < a$ and $F'(x) < \gamma_d <  \gamma_d$ when $x > b$. 
\end{proof}

\begin{proof}[Proof of proposition \ref{Theorem:Verification}]
The proof is a slight modification of the proof of theorem 3.6 in \cite{Lempa2014}. Define the almost surely finite stopping times $\tau:= \rho \wedge \tau_{\rho},$ where $\tau_{\rho} = \inf \{t \geq 0: X_t^{\zeta} \geq \rho \}$ and let $x \in \mathbb{R}_+$. Applying generalised Ito's lemma to the stopped process $e^{-r(t \wedge \tau)}F(X_{t \wedge \tau}^{\zeta})$ we get
\begin{equation}  \label{Eq:ItosFormula}
\begin{aligned}
    e^{-r(t \wedge \tau)}F(X_{t \wedge \tau}^{\zeta}) & = F(x) + \int_0^{t \wedge \tau} e^{-rs} (\mathcal{A}-r)F(X_s^{\zeta})ds \\ & + \int_0^{t \wedge \tau} e^{-rs} \sigma(X^{\zeta}_s)F'(X_s^{\zeta})dW_s \\ &+ \sum_{s \leq t \wedge \tau} e^{-rs}[F(X^{\zeta}_s)-F(X^{\zeta}_{s-})].
\end{aligned}
\end{equation}
Define the functions
\begin{align*}
    & \Phi_b(x) = \mathbbm{1}_{\{x > b\}} ((F(x)-\gamma_d x)-(F(b)-\gamma_d b)), \\
    & \Phi_a(x) = \mathbbm{1}_{\{x < a\}} ((F(x)-\gamma_u x)-(F(a)-\gamma_u a)).
\end{align*}
Because the control can jump only if the Poisson process jumps we have by lemma \ref{Lemma:CandidateValueProperties} that
\begin{align*}
     F(X^{\zeta}_s)-F(X^{\zeta}_{s-}) + \gamma_d \Delta \zeta^d_s \leq \lambda\Phi_b(X^{\zeta}_{s-}).
\end{align*}
if the jump is down. And similarly if the jump is up we have that
\begin{align*}
     F(X^{\zeta}_s)-F(X^{\zeta}_{s-}) -\gamma_u \Delta \zeta^u_s  \leq \lambda\Phi_a(X^{\zeta}_{s-}).
\end{align*}
Combining this with (\ref{Eq:ItosFormula}) we find
\begin{equation}  \label{Eq:ItosBoundedBelowF}
\begin{aligned}
    & e^{-r(t \wedge \tau)}F(X_{t \wedge \tau}^{\zeta}) +  \int_0^{t \wedge \tau} e^{-rs} (\pi(X_s^{\zeta})ds + \gamma_d d\zeta^d_s - \gamma_u d \zeta^u_s)  \\ 
    & \leq F(x)  +  \int_0^{t \wedge \tau} e^{-rs} \sigma(X^{\zeta}_s)F'(X_s^{\zeta})dW_s \\
    & +  \lambda \int_0^{t \wedge \tau} e^{-rs} \Phi_b(X_s^{\zeta})ds + \lambda \int_0^{t \wedge \tau} e^{-rs} \Phi_a(X_s^{\zeta})ds \\ 
    & + \int_0^{t \wedge \tau} e^{-rs} \Phi_b(X_s^{\zeta})dN_s +  \int_0^{t \wedge \tau} e^{-rs} \Phi_a(X_s^{\zeta})dN_s \\
    & = F(x) + M_{t \wedge \tau} + Z_{t \wedge \tau},
\end{aligned}
\end{equation}
where 
\begin{align*}
   & M_t = \int_0^{t \wedge \tau} e^{-rs} \sigma(X^{\zeta}_s)F'(X_s^{\zeta})dW_s,  \\
   & Z_t = \int_0^{t \wedge \tau} e^{-rs} \Phi_b(X_s^{\zeta})d \tilde{N}_s +  \int_0^{t \wedge \tau} e^{-rs} \Phi_a(X_s^{\zeta})d\tilde{N}_s
\end{align*}
are local martingales and the process $\tilde{N}_t = N_t - \lambda t$ is a compensated Poisson process. We notice from (\ref{Eq:ItosBoundedBelowF}) that the local martingale part $M_{t \wedge \tau} + Z_{t \wedge \tau}$ is bounded uniformly from below by $-F(x)$ and is thus a supermartingale. Taking the expectation in (\ref{Eq:ItosBoundedBelowF}) and letting $t,\rho \to \infty$ we find
\begin{equation*} 
\begin{aligned} 
    & \lim_{t,\rho \to \infty} \mathbb{E}_x[e^{-r(t \wedge \tau)}F(X_{t \wedge \tau}^{\zeta})] +  \\ & \mathbb{E}_x\bigg[\int_0^{\infty} e^{-rs} (\pi(X_s^{\zeta})ds + \gamma_d d\zeta^d_s - \gamma_u d \zeta^u_s) ds \bigg] 
   \leq F(x)
\end{aligned}
\end{equation*}
Hence, we conclude that
\begin{align*}
    F(x) \geq J(x, \zeta).
\end{align*}

To prove that the candidate value is attainable with admissible policy we show that $F(x) \leq J(x, \zeta^*)$. We first note that as the integrand in $M_{t \wedge \tau}$ is continuous and the stopped process $X^{\zeta^*}_{t \wedge \tau}$ is bounded and thus $M_{t \wedge \tau}$ is a martingale. Furthermore, 
\begin{align} \label{Eq:Zpart1}
    \int_0^{t \wedge \tau} e^{-rs} \Phi_b(X^{\zeta^*}_{s-})d \tilde{N}_s \leq \int_0^{t \wedge \tau} e^{-rs} F(X^{\zeta^*}_{s-}) \mathbbm{1}_{[b^*, \infty)}ds,
\end{align}
because $\Phi_b(x) \geq -F(x)$. Since $\pi_{\gamma_d}(x)$ and $\varphi_{r+\lambda}(x)$ are in $\mathcal{L}_1^r$ we observe by resolvent equation (\ref{Eq:ResolventEquation}) that (\ref{Eq:Zpart1}) is bounded uniformly from above by integrable random variable and consequently is a submartingale. Treating the other integral term in $Z_{t\wedge \tau}$ similarly, we see that $Z_{t\wedge \tau}$ is a submartingale.
As the inequality (\ref{Eq:ItosBoundedBelowF}) is equality for the proposed optimal control, we get by taking expectations that 
\begin{align*}
    F(x) \leq \liminf_{\tau, \rho \to \infty} \mathbb{E}_x[e^{-r(t \wedge \tau)} F(X^{\zeta^*}_{t\wedge \tau})] + J(x, \zeta^*).
\end{align*}
By lemma \ref{Lemma:CandidateValueProperties} we find that
\begin{align*}
    \mathbb{E}_x [e^{-r(t \wedge \tau(\rho))} F(X^{\zeta^*}_{t \wedge \tau(\rho)})] \leq \mathbb{E}_x [e^{-r(t \wedge \tau(\rho))} (F(b^*) + \gamma(X^{\zeta^*}_{t \wedge \tau(\rho)}-b^*))].
\end{align*}
Hence, we must have $\liminf_{\tau, \rho \to \infty} \mathbb{E}_x[e^{-r(t \wedge \tau)} F(X^{\zeta^*}_{t\wedge \tau})] = 0$, as otherwise we would contradict the assumption that $\text{id} \in \mathcal{L}_1^r$. Consequently, $V(x) = J(x, \zeta^*)$.
\end{proof}

\begin{remark}
We note that the above verification theorem and uniqueness of the optimal pair in proposition \ref{Prop:Uniqueness}, do not require most of our assumptions, and these are only needed to prove the existence of the solution. Thus, other type of assumptions could be allowed as long as one can be sure that the solution exists. Unfortunately, it seems very hard to find general conditions that cover interesting cases. A hint to this direction is given in p. 252 of \cite{Matomaki2012}, where the shape of the functions $\theta_n$ are somewhat relaxed, but then more restrictive assumptions about the boundary behaviour of the diffusion are needed.
\end{remark}

\section{Note on Asymptotics}

A similar control problem, where controlling is not restricted by a signal process is studied in \cite{Matomaki2012}. In this singular control case, we know that under the assumptions \ref{Ass:MainAssumptions}, the optimal control thresholds $(a_s^*, b_s^*)$ are the unique solution to the pair of equations
\begin{equation} \label{Eq:OptimalityConditionMatomaki}
\begin{split}
    &  L^r_{\theta_u}(a^*_s)=   L^r_{\theta_d}(b^*_s),\\
    &  K^r_{\theta_u}(a^*_s) =   K^r_{\theta_d}(b^*_s).
\end{split}
\end{equation}
Intuitively, this solution should coincide with ours when the signal rate $\lambda$ tend to infinity, because then the controlling opportunities are more and more frequent. The next proposition verifies this intuition.

\begin{proposition}
Let $K_{\lambda}(x)$ be as in (\ref{Eq:FunctionKDefinition}) and define a function $k: (0, a^*_s] \to (0, a^*_s]$ as (see \cite{Matomaki2012} p. 248)
\begin{equation*}
    k(x) = \check{L}^{r \, \, -1}_{\theta_u}( \hat{L}^r_{\theta_d} ( \hat{K}^{r \, \, -1}_{\theta_d}( \hat{K}^r_{\theta_u}(x)))),
\end{equation*}
where $\hat{\cdot}$ and $\check{\cdot}$ are restrictions to domains $[x^*_d, \infty)$ and $(0, x^*_u]$. Then the unique fixed point $a^*$ of $K_{\lambda}$ converges to the unique fixed point $a^*_s$ of $k$ as $\lambda$ tends to infinity.
\end{proposition}
\begin{proof}

For all $s > 0$, we have 
\begin{equation} 
\mathbb{E}_x[e^{-s\tau_z} \mid \tau_z < \infty] = \dfrac{\psi_{s}(z)}{\psi_{s}(x)},  \quad \,  z \leq x 
\end{equation}
where $\tau_z = \inf\{ t \geq 0 \mid X_t = z \}$. Further, by lemma \ref{Lemma:PhiPsiOrdering} the function $\lambda \mapsto \frac{\psi_{r+\lambda}(x)}{\psi_{r+\lambda}'(x)}$ is decreasing. Thus,
\begin{equation*}
    \frac{\psi_{r+\lambda}(z)}{\psi_{r+\lambda}'(x)} = \frac{\psi_{r+\lambda}(x)}{\psi_{r+\lambda}'(x)} \mathbb{E}_x[e^{-(r+\lambda)\tau_z} \mid \tau_z < \infty] \to 0, \quad  \text{when } \lambda \to \infty.
\end{equation*}
Consequently, by monotone convergence we find that 
\begin{equation}
    \frac{K_{\theta_u}^{r+\lambda}(x)}{\psi_{r+\lambda}'(x)} \to 0, \quad  \text{when } \lambda \to \infty.
\end{equation}
Similarly we can show that 
\begin{equation}
    \frac{L_{\theta_d}^{r+\lambda}(x)}{\varphi_{r+\lambda}'(x)} \to 0, \quad  \text{when } \lambda \to \infty.
\end{equation}
Hence, we observe from the representation of the pair of equations (\ref{Eq:PairRepresented}) and monotonicity that 
\begin{equation*}
\lim_{\lambda \to \infty} K_{\lambda}(x) \xrightarrow{\lambda \to \infty} k(x).
\end{equation*}
\end{proof}
Also, a rather straightforward calculation of the limit $\lambda \to 0$ in $(\ref{Eq:CandidateValue})$ yields $V(x) = (R_r \pi)(x)$ for all $x \in \mathbb{R}_+$. This result corresponds to the case where the signal process does not jump at all and thus there is no opportunities to control the underlying process. Hence, the reward that the controller gets is the resolvent $(R_r \pi)(x)$, i.e. the expected cumulative present value of the instantaneous payoff $\pi$.

\section{Illustration: geometric Brownian motion}

We assume that the underlying diffusion is a standard geometric Brownian motion and thus the infinitesimal generator reads as
\begin{align*}
    \mathcal{A} = \frac{1}{2}\sigma^2 x^2\frac{d^2}{dx^2} + \mu x \frac{d}{dx},
\end{align*}
where $\mu \in \mathbb{R}_+$ and $\sigma \in \mathbb{R}_+$ are given constants. Furthermore, the scale density and the density of the speed measure read as
$$
S'(x) = x^{-\frac{2 \mu}{\sigma^2}}, \quad 
m'(x) = \frac{2}{\sigma^2} x^{\frac{2 \mu}{\sigma^2}-2}.
$$
Assume that $\mu < r$ and denote
\begin{align*}
& \beta_{\lambda} = \frac{1}{2}-\frac{\mu}{\sigma^2}+\sqrt{\bigg(\frac{1}{2}-\frac{\mu}{\sigma^2} \bigg)^2+\frac{2(r+\lambda)}{\sigma^2}} > 1, \\
& \alpha_{\lambda} = \frac{1}{2}-\frac{\mu}{\sigma^2} - \sqrt{\bigg(\frac{1}{2}-\frac{\mu}{\sigma^2} \bigg)^2+\frac{2(r+\lambda)}{\sigma^2}} < 0.
\end{align*}
Then the minimal $r$-excessive functions for $X$ read as $$\psi_{r}(x) = x^{\beta_0}, \varphi_{r}(x) = x^{\alpha_0}, \psi_{r+\lambda}(x) = x^{\beta_{\lambda}}, \varphi_{r+\lambda}(x) = x^{\alpha_{\lambda}}.$$

Define the instantaneous payoff by $\pi(x) = x^{\delta}, \, \, 0 < \delta < 1$. Then the net convenience yield is given by $\theta_n(x) = x^\delta -\gamma_n (r-\mu)x$, where $n \in \{u,d\}$ and $\gamma_d < \gamma_u$. We readily verify that our assumptions \ref{Ass:MainAssumptions} hold in this case. The resolvent reads as
$$
(R_r \pi)(x) = \frac{x^{\delta}}{r-\delta\mu- \frac{1}{2}\sigma^2\delta(\delta-1)}.
$$
Noting that $(\beta_0-\delta)(\delta-\alpha_0) = 2(r-\delta \mu- \frac{1}{2}\sigma^2 \delta(\delta-1))/\sigma^2$, we find the alternative representation
$$
(R_r \pi)(x) = \frac{2}{\sigma^2}\frac{x^{\delta}}{(\beta_0-\delta)(\delta-\alpha_0)}.
$$

To write down the pair of equations we first calculate the auxiliary functionals 
\begin{align*}
    K_{\theta_u}^{r+\lambda}(x) & = \frac{2(r+\lambda)}{\sigma^2} \bigg( \bigg[ \frac{1}{\delta-\alpha_{\lambda}}+\frac{1}{\alpha_{\lambda}}\bigg] x^{\delta-\alpha_{\lambda}}+\bigg[\frac{\gamma_u(\mu-r)}{1-\alpha_{\lambda}}+\frac{\gamma_u(\mu-r)}{\alpha_{\lambda}} \bigg] x^{1-\alpha_{\lambda}} \bigg), \\
    L_{\theta_d}^{r+\lambda}(x) & = -\frac{2(r+\lambda)}{\sigma^2} \bigg( \bigg[ \frac{1}{\delta-\beta_{\lambda}}+\frac{1}{\beta_{\lambda}}\bigg] x^{\delta-\beta_{\lambda}}-\bigg[\frac{\gamma_d(\mu-r)}{1-\beta_{\lambda}}+\frac{\gamma_d(\mu-r)}{\beta_{\lambda}} \bigg] x^{1-\beta_{\lambda}} \bigg).
\end{align*}
Then noting that $\beta_0-\beta_{\lambda}=\alpha_{\lambda}-\alpha_0$ and that $(r+\lambda)\beta_0 \alpha_0 = r \beta_{\lambda} \alpha_{\lambda}$, we find that the pair of equations read as 
\begin{equation}\label{Eq:GBMpair}
    \begin{aligned} 
      & \bigg[ \frac{\alpha_{\lambda}}{\alpha_0(\delta-\alpha_{\lambda})}- \frac{1}{\delta-\alpha_{0}} \bigg] a^{\delta-\alpha_{0}}+\bigg[\frac{\gamma_u(\mu-r)\alpha_{\lambda}}{\alpha_0(1-\alpha_{\lambda})}-\frac{\gamma_u(\mu-r)}{1-\alpha_{0}} \bigg] a^{1-\alpha_{0}} \\ 
      = & \bigg[ \frac{\beta_{\lambda}}{\alpha_0(\delta-\beta_{\lambda})}- \frac{1}{\delta-\alpha_{0}} \bigg] b^{\delta-\alpha_{0}}+\bigg[\frac{\gamma_d(\mu-r)\beta_{\lambda}}{\alpha_0(1-\beta_{\lambda})}-\frac{\gamma_u(\mu-r)}{1-\alpha_{0}} \bigg] b^{1-\alpha_{0}} , \\
      & \bigg[ \frac{\alpha_{\lambda}}{\beta_0(\delta-\alpha_{\lambda})}- \frac{1}{\delta-\beta_{0}} \bigg] a^{\delta-\beta_{0}}+\bigg[\frac{\gamma_u(\mu-r)\alpha_{\lambda}}{\beta_0(1-\alpha_{\lambda})}-\frac{\gamma_u(\mu-r)}{1-\beta_{0}} \bigg] a^{1-\beta_{0}} \\ 
      = & \bigg[ \frac{\beta_{\lambda}}{\beta_0(\delta-\beta_{\lambda})}- \frac{1}{\delta-\beta_{0}} \bigg] b^{\delta-\beta_{0}}+\bigg[\frac{\gamma_d(\mu-r)\beta_{\lambda}}{\beta_0(1-\beta_{\lambda})}-\frac{\gamma_u(\mu-r)}{1-\beta_{0}} \bigg] b^{1-\beta_{0}}.
\end{aligned}
\end{equation}
However, the assumption $\tilde{x}<\hat{x}$ in proposition \ref{Prop:Existence} has to be analyzed separately. We find that $\tilde{x}$ and $\hat{x}$ read as
\begin{equation*}
    \tilde{x} = \gamma_u(r-\mu)\frac{\delta-\alpha_0}{\delta(1-\alpha_0)}, \quad \, \, \hat{x} = \gamma_d(r-\mu)\frac{\delta-\beta_0}{\delta(1-\beta_0)}.
\end{equation*}
Hence, the condition $\tilde{x}<\hat{x}$ is equivalent to
\begin{equation*}
   \gamma_u \frac{\delta-\alpha_0}{1-\alpha_0} < \gamma_d \frac{\delta-\beta_0}{1-\beta_0}.
\end{equation*}
Unfortunately, it seems impossible to solve the pair (\ref{Eq:GBMpair}) explicitly and thus we illustrate the results numerically. We choose the parameters
$\mu = 0.05,
\sigma = 0.2,
r = 0.15,
\lambda=2,
\gamma_u = 5,
\gamma_d = 4,
\delta = 0.3.
$ Then, as expected, the optimal thresholds converge to the ones in the singular case as $\lambda \to \infty$, see Figure \ref{Fig:Lambdalim}.
\begin{figure}[ht]
    \centering
    \includegraphics[width=\linewidth]{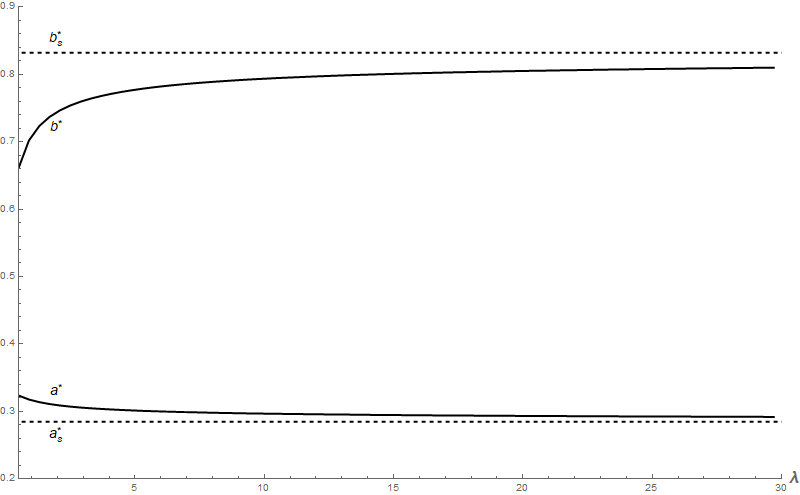}
    \caption{The optimal boundaries $a^*_s,$ $b^*_s$, $a^*$ and $b^*$ when parameters $\mu = 0.05,\sigma = 0.2,r = 0.15,\gamma_u = 5,\gamma_d = 4,\delta = 0.3.$ and $\lambda$ increases.}
    \label{Fig:Lambdalim}
\end{figure}
Also, at least in our numerical examples, increasing volatility expands the continuation region (by increasing $b$ and decreasing $a$), see table \ref{Tab:Incuncer}.
\begin{table}[ht]
    \centering
    \begin{tabular}{|c|c|c|c|}
             \hline & $\sigma = 0.2, \lambda=2$ & $\sigma = 0.8, \lambda=2$ & $\sigma = 0.8, \lambda=20$ \\ \hline
         $a^*$ & $0.309$  & $0.188$ & $0.151$   \\ \hline
         $b^*$ & $0.745$ & $0.938$ & $1.246$    \\ \hline
    \end{tabular}
    \caption{The optimal boundaries $a^*$ and $b^*$ when parameters $\mu = 0.05,r = 0.15,\gamma_u = 5,\gamma_d = 4$ and $\delta = 0.3.$}
    \label{Tab:Incuncer}
\end{table}
This is in line with the findings for the singular control case. These observations show that on one hand increased uncertainty (in terms of decreasing signal rate $\lambda$) shrinks the inactivity region by hastening the usage of control policies, but on the other hand increased uncertainty (in terms of increasing volatility $\sigma$) expands the inactivity region.

Finally, the value function of the problem is shown in Figure \ref{Fig:Valuefunc}.
\begin{figure}[ht]
    \centering
    \includegraphics[width=\linewidth]{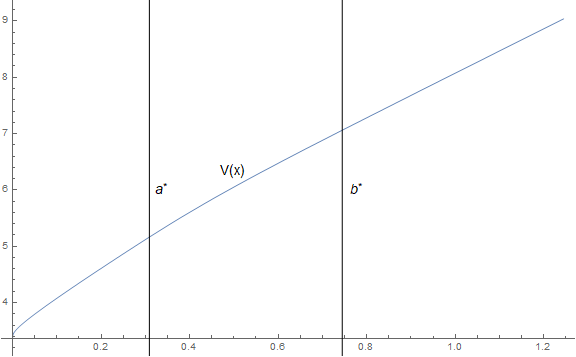}
    \caption{The value function and the optimal boundaries $a^*$ and $b^*$ with parameters $\mu = 0.05,\sigma = 0.2,r = 0.15,\lambda=2,\gamma_u = 5,\gamma_d = 4,\delta = 0.3.$}
    \label{Fig:Valuefunc}
\end{figure}


\begin{thebibliography}{99}

%\bibitem{Alvarez1999} 
%{\sc Alvarez, L. H. R.} (1999). A Class of Solvable Singular Stochastic Control Problems. {\em Stochastics and Stochastics Reports}
%{\bf 67,} 83--122.

\bibitem{Alvarez2001} 
{\sc Alvarez, L. H. R.} (2001). Singular Stochastic Control, Linear Diffusions, and Optimal Stopping: A Class of Solvable Problems, {\em SIAM J. Control Optim.}
{\em 39,} 1697--1710.

\bibitem{Alvarez2003}
{\sc Alvarez, L. H. R.} (2003). On the properties of r-excessive mappings for a class of diffusions, {\em Ann. Appl. Probab.}
{\em 13,} 1517--1533.

\bibitem{Alvarez2004} 
{\sc Alvarez, L. H. R.} (2004). A Class of Solvable Impulse Control Problems, {\em Appl. Math. Optimization}
{\em 49,} 265--295.

\bibitem{AlvarezLempa2008} 
{\sc Alvarez, L. H. R. and Lempa, J.} (2008). On The Optimal Stochastic Impulse Control of Linear Diffusions. {\em SIAM J. Control Optim.}
{\bf 47,} 703--732.

\bibitem{AlvarezHening2019} 
{\sc Alvarez, L. H. R. and Hening, A.} (2019). Optimal sustainable harvesting of populations in random environments.  To appear in {\em Stochastic Processes and their Applications}.

\bibitem{BorodinSalminen} 
{\sc Borodin, A. N. and Salminen, P.} (2015). {\em Handbook of Brownian Motion - Facts and Formulae}, 2nd~edn. Birkhäuser, Basel.

\bibitem{ChiarollaHaussman2005} 
{\sc Chiarolla M.B. and Haussmann U.G.} (2005). Explicit solution of a stochastic irreversible investment problem and its moving threshold. {\em Math. Oper. Res.}
{\bf 30,} 91--108.

\bibitem{Dixit1990} 
{\sc Dixit, A. K. and Pindyck, R. S.} (1994). {\em Investment under uncertainty}, Princeton UP, Princeton. 

\bibitem{DupuisWang2002} 
{\sc Dupuis, P. and Wang, H.} (2002). Optimal Stopping with Random Intervention Times. {\em Adv. Appl. Prob.} 
{\bf 34,} 141--157.  

\bibitem{GuoPham2005} 
{\sc Guo X. and Pham H.} (2005). Optimal partially reversible investment with entry decision and general production function. {\em  Stochast. Process. Appl.} 
{\bf 115,} 705--736.  

\bibitem{HarrisonSellkeTaylor1983} 
{\sc Harrison, J. M., Sellke, T. M. and Taylor, A. J.} (1983). Impulse Control of Brownian Motion. {\em Math. Oper. Res.}
{\bf 8,} 454--466.

\bibitem{HobsonZeng2019} 
{\sc Hobson D. and Zeng M.} (2019). Constrained optimal stopping, liquidity and effort. {\em To appear in Stochastic Processes and Applications.} {\em preprint, arXiv:2008.01787.}

\bibitem{Hobson2021} 
{\sc Hobson D.} (2021). The shape of the value function under Poisson optimal stopping. {\em Stochast. Process. and Appl.}
{\bf 133,} 229--246.  

%\bibitem{Karatzas1983} 
%{\sc Karatzas, I.} (1983). A class of Singular Stochastic Control Problems. {\em Adv. Appl. Prob.}
%{\bf 15,} 225--254.

\bibitem{KaratzasShreve1991} 
{\sc Karatzas, I. and Shreve, S. E.} (1991). {\em Brownian Motion and Stochastic Calculus}. Springer, New York.

\bibitem{Lange2020} 
{\sc Lange, R.-J., Ralph D. and Støre K.} (2020). Real-option valuation in multiple dimensions using Poisson optional stopping times. {\em J. Finan. Quant. Anal.}
{\bf 55,} 653--677.

\bibitem{Lempa2010} 
{\sc Lempa, J.} (2010). A note on optimal stopping of diffusions with a two-sided optimal rule. {\em Oper. Res. Lett.}
{\bf 38,} 11--16.

\bibitem{Lempa2012} 
{\sc Lempa, J.} (2012). Optimal Stopping with Information Constraint. {\em Appl. Math. Optimization}
{\bf 66,} 147--173.

\bibitem{Lempa2014} 
{\sc Lempa, J.} (2014). Bounded Variation Control of Itô Diffusion with Exogenously Restricted Intervention Times. {\em Adv. Appl. Prob.}
{\bf 46,} 102--120.

\bibitem{LempaSaarinen2021} 
{\sc Lempa, J. and Saarinen, H.} (2021). Ergodic Control of Diffusions with Random Intervention Times. {\em J. Appl. Probab.}.
{\bf 58,} 1--21.

\bibitem{LiangWei2016}
Liang, G. and W. Wei (2016). Optimal Switching at Poisson Random Intervention Times. {\em Discrete and Continuous Dynamical Systems-Series B.} 
{\bf 21,} 1483--1505.

\bibitem{LiangSun2019}
{\sc Liang, G. and Sun H.} (2019). Dynkin games with Poisson Random Intervention Times. {\em SIAM Journal on Control and Optimization.}
{\bf 57,} 2962-–2991.

\bibitem{LiangSun2020}
{\sc Liang, G. and Sun H.} (2020). Risk-Sensitive Dynkin Games with Heterogeneous Poisson Random Intervention Times. {\em preprint, arXiv:2008.01787.} 

\bibitem{LunguOksendal1997}
{\sc Liang, G. and Sun H.} (2020). Optimal harvesting from a population in a stochastic crowded environment. {\em Math. Biosci.}
{\bf 145,} 47-–75

\bibitem{Matsumoto2006} 
{\sc Matsumoto, K.} (2006). ptimal Portfolio of Low Liquid Assets with a Log-utility Function. {\em Finance Stochastics.}
{\bf 10,} 121--145.

\bibitem{Matomaki2012} 
{\sc Matomäki P.} (2012). On Solvability of a Two-sided Singular control Problem. {\em Math. Oper. Res.}
{\bf 76,} 239--271.

\bibitem{MenaldiRobin2016} 
{\sc Menaldi, J. L. and Robin, M.} (2016). On Some Optimal Stopping Problems with Constraint. {\em SIAM J. Control Optim.} 
{\bf 54,} 2650--2671.

\bibitem{MenaldiRobin2017} 
{\sc Menaldi, J. L. and Robin, M.} (2017). On Some Impulse Control Problems with Constraint. {\em SIAM J. Control Optim.}
{\bf 55,} 3204--3225.

\bibitem{MenaldiRobin2018} 
{\sc Menaldi, J. L. and Robin, M.} (2018). On Some Ergodic Impulse Control Problems with Constraint. {\em SIAM J. Control Optim.}
{\bf 56,} 2690--2711.

\bibitem{Oksendal2000} 
{\sc Øksendal A.} (2000). Irreversible investment problems. {\em Finance Stochastics.}
{\bf 4,} 223--250.


\bibitem{RogersZane2002} 
{\sc Rogers, L. C. G. and Zane, O.} (2002). A Simple Model of Liquidity Effects. In {\em Advances in Finance and Stochastics}, 161--176.  Springer-Verlag, Berlin Heidelberg.

%\bibitem{SethiZhang2005} 
%{\sc Sethi, S. P., Zhang, H. and Zhang, Q.} (2005). {\em Average-Cost Control of Stochastic Manufacturing Systems}. Springer-Verlag, New York.

%\bibitem{Yor1997} 
%{\sc Revuz, D and Yor, M} (1991). {\em Continuous Martingales and Brownian Motion}. Springer-Verlag, Berlin Heidelberg. 

\bibitem{Wang2001} 
{\sc Wang, H.} (2001). Some Control Problems with Random Intervention Times. {\em Adv. Appl. Prob.}
{\bf 33,} 404--422.

\end{thebibliography}
\end{document}